\documentclass[preprint,12pt]{article}

\usepackage{amssymb}

\usepackage{amsmath,amssymb,amsfonts,esint}

\newcommand\dps{\displaystyle }

 \usepackage{latexsym}
 \usepackage{pdflscape}
 \usepackage{examplep}
 \usepackage{longtable}
 \usepackage{graphicx}
 \usepackage{amsmath}
 \usepackage{booktabs}
 \usepackage{memhfixc}
 \usepackage{dsfont}
 \usepackage{amsmath,amssymb,amsfonts,amsthm}
 \usepackage[francais,english]{babel}

\newtheorem{prop}{Proposition}

\def\N{\mathbb{N}}
\def\R{\mathbb{R}}
\def\cL{\mathcal{L}}

\def\cC{\mathcal{C}}
\def\bfa{\textbf{a}}
\def\bfb{\textbf{b}}
\def\bfc{\textbf{c}}
\def\bfl{\textbf{l}}
\def\rt{\widetilde{r}}
\def\st{\widetilde{s}}
\def\dps{\displaystyle}

\title{A dynamical adaptive tensor method for the Vlasov-Poisson system}

\author{Virginie Ehrlacher\thanks{Universit\'e Paris Est, CERMICS, Projet Matherials, Ecole des Ponts ParisTech - INRIA, 6 \& 8 avenue Blaise Pascal, 77455 Marne-la-Vall\'ee Cedex 2, France, ({\tt ehrlachv@cermics.enpc.fr})}  \and Damiano Lombardi\thanks{INRIA Paris, ({\tt damiano.lombardi@inria.fr})}}

\begin{document}
\maketitle

\begin{abstract}
A numerical method is proposed to solve the full-Eulerian time-dependent Vlasov-Poisson system in high dimension. The algorithm relies on the construction of a tensor decomposition of 
the solution whose rank is adapted at each time step. This decomposition is obtained through the use of an efficient modified Progressive Generalized Decomposition (PGD) method, whose 
convergence is proved. We suggest in addition a symplectic time-discretization splitting scheme that preserves the Hamiltonian properties of the system. This scheme is naturally 
obtained by considering the tensor structure of the approximation. The efficiency of our approach is illustrated through time-dependent 2D-2D numerical examples.  
\end{abstract}


\section{Introduction}
The present work investigates a numerical method for the resolution of the time-dependent Vlasov-Poisson system. The solution is approximated using parsimonious tensor methods.

\medskip

In the litterature, equations arising in kinetic theory are solved by three classes of approaches: particle methods (Particle-In-Cell~\cite{germaschewski2016,brackbill2016,cazeaux2014multiscale}, 
Particle-In-Cloud~\cite{wang2016}), semi-lagrangian approaches~\cite{crouseilles2009,charles2013,kormann2015,pinto2008,crouseilles2010} and full-deterministic Eulerian methods~\cite{filbet2003,madaule2014,xu2010scalable}.
In this work, we focus on the Vlasov-Poisson system as a simple yet challenging example of kinetic equation. While Eulerian approaches are appealing to describe the evolution of the unknown quantities of interest,
the high dimensionality of the phase space domain make them often prohibitive in terms of memory and computational cost, especially when 2D-2D and 3D-3D problems are at hand.

The proposed method is not a particular discretization \emph{per se}, instead, it gives a way to build a parsimonius tensor decomposition starting from chosen a priori separated discretizations 
for the space and the velocity variables in a full eulerian approach. The contribution is twofold: first, we show that the use of a tensorised representation of the solution induces a 
natural splitting of the equations which respects the Hamiltonian nature of the Vlasov-Poisson equations; second, an efficient fixed-point algorithm is proposed 
to solve the (non-symmetric) equations using tensorised functions. This step is performed using a modified Proper Generalized Decomposition (PGD) 
method~\cite{chinesta2011,figueroa2012,chinesta2010,falco2012,falco2011,cances2011,le2009results}, and the convergence of the scheme is proved. 
Let us mention that close ideas were introduced in the recent work~\cite{cho2016} for the evolution of high-dimensional probability densities. In the contribution~\cite{kormann2015}, 
a tensor train method is used to discretize 
the Vlasov-Poisson equations by separating each component, in a semi-lagrangian approach. 

Here, we do not separate in all the variables in order to deal with generic space (and possibly velocity) domain geometries~\cite{xu2010scalable}. 
Thus, only second order tensors are used. The proposed method dynamically adapts through time the rank of the decomposition. This is an important feature, as was noted in~\cite{cho2016,kormann2015}, since 
the number of tensorised terms needed to approximate with a given tolerance the solution at a certain time is not known \emph{a priori}.

The structure of the work is as follows: in Section~\ref{sec:hamiltonian}, 
the Vlasov-Poisson system is recalled in its classical and Hamiltonian formulation. 
A discussion on how a tensorised representation leads to a natural splitting of the evolution is presented in Section~\ref{sec:splitting}.

A second-order symplectic scheme in time is derived for the tensor representation update. Then, in Section~\ref{sec:tensor}, after a brief review of the PGD method for the resolution of symmetric coercive problems, a fixed-point scheme is presented, 
to solve some non-symmetric linear problems arising in the Vlasov-Poisson context. The proof of convergence of the algorithm is presented in Appendix~A. 
Numerical tests illustrating the properties of the method are presented in Section~\ref{sec:numtest}.

\section{Hamiltonian formulation and tensor decomposition}\label{sec:hamiltonian}
In this section, the Hamiltonian formulation of the Vlasov-Poisson system is recalled. 
A particular emphasis is put on the elements that play an important role in the derivation of the proposed numerical method. 
The idea is to compute a tensor decomposition of the solution of the Vlasov-Poisson system and to use a symplectic integrator in time
in order to preserve the hamiltonian structure of the equations. As it will be shown in Section~\ref{sec:splitting}, the tensorised expansion induces a natural splitting of the equations. 

\subsection{The Vlasov-Poisson system}


Let $d\in\N^*$ denote the spatial dimension of the problem and $\Omega_x, \Omega_v \subseteq \R^d$. The Vlasov-Poisson system for negative electric charges reads:
\begin{eqnarray}
\label{eq:VlasovPoisson}
 \partial_t f + v\cdot\nabla_x f - E\cdot\nabla_v f = 0, & \quad \mbox{ in } (0,+\infty) \times \Omega_x \times \Omega_v,\nonumber \\
 -\Delta_x \varphi  = 1 - \int_{\Omega_v} f \ dv,&  \quad \mbox{ in } (0,+\infty) \times \Omega_x, \nonumber \\
E = -\nabla_x \varphi, & \quad  \mbox{ in } (0,+\infty) \times \Omega_x, \nonumber \\
f(0,x,v) = f_0(x,v), & \quad  \mbox{ in } \Omega_x \times \Omega_v,\\ \nonumber
\end{eqnarray} 
with appropriate boundary conditions on $\Omega_x \times \Omega_v$, where 
$$
f: \left\{ 
	    \begin{array}{ccc}
	    (0,+\infty)\times \Omega_x \times \Omega_v & \to & \R\\
	    (t,x,v) & \mapsto & f(t,x,v)\\ 
	    \end{array}\right.
$$ is the particle distribution function in the phase space, $f_0\geq 0$ is the initial particle distribution function, $E(t,x)$ the electric fied and $\varphi(t,x)$ the electric potential. 
The particle density $\rho(t,x)$ is given by $\rho(t,x) = \int_{\Omega_v} f(t,x,v)\,dv$, and hence, the equation for the electric potential reads $-\Delta_x \varphi= 1 - \rho$. 

\medskip

The global existence of positive (weak or strong) solutions has been studied in several works~\cite{arsenev1975,BardosDegond, Dolbeault, LionsPerthame, glassey1996,hwang2004, ambrosio2014}.

For instance, in~\cite{LionsPerthame}, 
when $\Omega_x = \Omega_v = \R^3$, the existence of a strong non-negative solution $f\in \cC(\R_+ ; L^1(\R^3\times \R^3)) \cap L^\infty (\R_+ \times \R^3 \times \R^3)$ is proved provided that the initial condition 
$f_0 \in L^1\cap L^\infty(\R^3 \times \R^3)$ satisfies the additional condition: for some $m_0 > 3$, 
$$
\int_{\R^3\times \R^3} f_0(x,v) \mid v  \mid^{m_0}\,dx\,dv < +\infty.
$$

\medskip

For numerical purposes, equation (\ref{eq:VlasovPoisson}) has to be solved on a truncated domain $\Omega = \Omega_x \times \Omega_v$. One can for instance impose periodic boundary conditions on $\partial \Omega_x$ 
or $\partial \Omega_v$, as done in~\cite{madaule2014}. An alternative formulation would consist of imposing homogeneous boundary conditions on $\Omega_v$, the velocity domain.

\medskip

Extensive reviews on the theory and numerical methods for this type of kinetic equations are detailed in~\cite{vedenyapin2011,degond2004}.

\subsection{Hamiltonian formulation}
The system (\ref{eq:VlasovPoisson}) may be derived by using an Hamiltonian formalism (see \cite{MARSDEN}).

The Hamiltonian for the Vlasov-Poisson system reads:
\begin{equation}
\mathcal{H} = \int_{\Omega} \frac{1}{2} f \mid v \mid^2 \ dx \ dv - \int_{\Omega_x} \frac{1}{2} \varphi \rho \ dx.
\end{equation}
The first term in the Hamiltonian is the kinetic energy of the particles, and the second term accounts for the electro-static energy. As commented in \cite{MARSDEN}, the Vlasov-Poisson equations can be derived by introducing a reduced Poisson bracket:
\begin{equation}
\left\lbrace a,b \right\rbrace := \nabla_x a \cdot \nabla_v b - \nabla_v a \cdot \nabla_x b.
\end{equation}

The evolution equation for the system can thus be written as:
\begin{equation}
\partial_t f = -\left\lbrace f,h \right\rbrace,
\label{eq:hamVlasov}
\end{equation}
where $h:=\frac{1}{2} \mid v \mid^2 - \varphi$ is the reduced Hamiltonian. A precise and detailed derivation of the Hamiltonian structure of the Vlasov-Poisson equations 
is found in \cite{MARSDEN, morrison2005}. 

\subsection{Splitting induced by tensor decomposition}\label{sec:splitting} 

For any measurable functions $r:\Omega_x \to \R$ and $s:\Omega_v \to \R$, we define the tensor product function $r\otimes s: \Omega_x \times \Omega_v \to \R$ as
$$
r\otimes s: \left\{
\begin{array}{ccc}
 \Omega_x \times \Omega_v & \to & \R\\
 (x,v) & \mapsto & r(x)s(v).\\
\end{array} \right.
$$
In the sequel, such a function is referred to as a {\itshape pure tensor-product} function. A linear combination of $n$ pure tensor-product functions (for some $n\in\N^*$) is called a {\itshape rank-}$n$ tensor product function.

\medskip

We also introduce here the notion of tensorized operators. Let $H_x$ (respectively $H_v$) be a Hilbert space of real-valued functions defined on $\Omega_x$ (respectively on $\Omega_v$) and $H$ a Hilbert space 
of functions defined on $\Omega_x \times \Omega_v$ so that $H_x \otimes H_v \subseteq H$.
An operator $A$ acting on functions depending on both $x$ and $v$ variables is a {\itshape tensorized operator} if it can be written as 
$$
A = \sum_{\lambda=1}^L A_x^{\lambda} \otimes A_v^{\lambda}, 
$$
for some $L\in\N^*$,  where for all $1\leq \lambda \leq L$, $A_x^{\lambda}$ (respectively $A_v^{\lambda}$) is an operator on $H_x$ (respectively on $H_v$). Let us remind the reader that for 
all operators $A_x$ on $H_x$, $A_v$ on $H_v$, and $(r,s)\in H_x \times H_v$, 
$$
\left( A_x \otimes A_v \right) (r\otimes s) = (A_x r) \otimes (A_v s).
$$

\medskip
In this section, a formal calculation is presented, which justifies how such a decomposition induces a natural splitting of the Vlasov-Poisson equations. 
Let us mention here the work~\cite{casas2015}, where a high-order splitting of the hamiltonian formulation for the Vlasov-Maxwell system was recently proposed.

In the present method, the aim is to approximate the function $f$, solution of (\ref{eq:VlasovPoisson}), 
by a separate variate expansion of the form:
\begin{equation}
f(x,v,t) \approx \sum_{k=1}^n r_k(x,t) s_k(v,t) = \sum_{k=1}^n r_k(\cdot,t)\otimes s_k(\cdot,t),
\end{equation}
with some measurable functions $r_k : \Omega_x \times \R_+ \to \R$, $s_k: \Omega_v  \times \R_+ \to \R$ and $n\in\N^*$. 
When this expression is inserted into the evolution equation written in a hamiltonian form, it reads:
\begin{equation}
\partial_t f = - \left\lbrace f,h \right\rbrace \approx \sum_{k=1}^n -\left\lbrace r_k , h \right\rbrace s_k  -r \left\lbrace s_k , h \right\rbrace.
\end{equation}
The Poisson bracket acting on $r_k$ and $s_k$, separately, can be interpreted as the operator which is inducing a dynamics on the functions $r_k$ and $s_k$. Indeed, 
when considering $\partial_t r_k = -\left\lbrace r_k , h \right\rbrace$ and $\partial_t s_k = -\left\lbrace s_k , h \right\rbrace$, the tensor decomposition implies 
naturally $\partial_t f = -\left\lbrace f , h \right\rbrace$. Consider a particular time $t=t^*$, for which $f(x,v,t^*) \approx \sum_{k=1}^n r_k(x,t^*) s_k(v,t^*)$. The action of the Poisson bracket on generic functions $r(x)$ and $s(v)$ depending respectively only upon the space coordinate $x\in \Omega_x$ or the velocity $v\in \Omega_v$ reads:
\begin{eqnarray}
\left\lbrace r, h\right\rbrace = v\cdot \nabla_x r(x), \\
\left\lbrace s, h\right\rbrace = \nabla_x \varphi \cdot \nabla_v s(v).
\end{eqnarray}
Two facts are fundamental: first, the evolution operator splits naturally into two parts, one acting on $r$ and the other on $s$. Second, the evolution of each part is the action of a tensorised operator acting on the functions.

\subsection{Symplectic integrator in time}\label{sec:scheme}

Let us define $\widetilde{r}_k: (t^*, T) \times \Omega_x \times \Omega_v \to \R$ and $\widetilde{s}_k: (t^*, T) \times \Omega_x \times \Omega_v \to \R$ as solutions to the dynamical system: 
\begin{eqnarray}
& \partial_t \rt_k(x,v,t^*) = - \{ \rt_k, h\}, \nonumber \\ 
& \partial_t \st_k(x,v,t^*) =  - \{ \st_k, h\}, \nonumber\\
& \rt_k(x,v,t^*) = r_k(x,t^*), \quad \st_k(x,v,t^*)  = s_k(v, t^*).
\label{eq:ev_rt_st}
\end{eqnarray}
At $t = t^*$, it holds that
\begin{eqnarray}
\partial_t \rt_k(x,v,t^*) = - v \cdot \nabla_x r_k(x,t^*), \nonumber \\ 
\partial_t \st_k(x,v,t^*) = - \nabla_x \varphi(x,t^*) \cdot \nabla_v s_k(v, t^*).
\label{eq:xv_var}
\end{eqnarray}
In other words, the time derivative of a generic element $\rt_k$ computed at time $t=t^*$ is given 
by the advection part of the Vlasov-Poisson system, whereas the time derivative of the element $\st_k$ 
is given by the electrostatic force. Remark that initial functions $r_k$ at time $t=t^*$ depend only upon $x$, 
but the time derivative depends of course also on $v$. The analogue is true for the $s_k$ functions. 

\medskip

A symplectic discretization in time for the system~(\ref{eq:hamVlasov}) is proposed, based on this remark. 
For a comprehensive overview of geometric integrators see \cite{lubich2006}. Let $\Delta t >0$ be a small time step. 
The starting point is to consider the system (\ref{eq:xv_var}) and use a St\"ormer-Verlet algorithm (see \cite{lubich2003}) 
to discretize the evolution of the functions $\rt_k$ and $\st_k$ between times $t^*$ and $t^* + \Delta t$. This scheme is obtained 
by considering $\rt_k$ and $\st_k$ as if they were the coordinates and the momenta of the Hamiltonian system associated to the Vlasov-Poisson equation. 
Define $\st_k^{\:(t^*)}(x,v) := \st_k(x, v, t^*) = s_k(v,t^*)$, $\rt_k^{\:(t^*)}(x,v) := \rt_k(x, v, t^*) =  r_k(x,t^*)$, and $\st_k^{\:(t^*+\Delta t/2)}$, $\rt_k^{\:(t^* + \Delta t)}$ and 
$\st_k^{\:(t^* + \Delta t)}$ as follows:
\begin{eqnarray}
\st_k^{\:(t^*+\Delta t/2)}(x,v) = \st_k^{\:(t^*)}(x,v) + \frac{\Delta t}{2} E^{\:(t^*)}(x)\cdot \nabla_v \st_k^{\:(t^* + \Delta t/2)}(x,v), \\
\rt_k^{\:(t^* + \Delta t)}(x,v) = \rt_k^{\:(t^*)}(x,v)  - \frac{\Delta t}{2} \left( v\cdot \nabla_x \rt_k^{\:(t^*)}(x,v) + v\cdot \nabla_x \st_k^{\:(t^* + \Delta t/2)}(x,v) \right), \\
\st_k^{\:(t^* + \Delta t)}(x,v) = \st_k^{\:(t^*+ \Delta t/2)}(x,v) + \frac{\Delta t}{2} E^{\:(t^*+2\Delta t/3)}(x) \cdot \nabla_v \st_k^{\:(t^* + \Delta t/2)}(x,v),
\end{eqnarray}
where the definitions of the electric fields are given below. 
Remember that $f(x,v,t^*) = \sum_{k=1}^n r_k(x,t)s_k(t,v)$. We define $E^{\:(t^*)}$ and $E^{\:(t^*+ 2\Delta t/3)}$ by
\begin{eqnarray}
\rho^{\:(t^*)}(x) := \int_{\Omega_v}f(x,v,t^*)\ dv, \nonumber \\
-\Delta_x \varphi^{\:(t^*)}(x) = 1 - \rho^{\:(t^*)}(x), \\
E^{\:(t^*)}(x) = -\nabla_x \varphi^{\:(t^*)}(x), \nonumber
\end{eqnarray}
and
\begin{eqnarray}
f^{\:(t^*+2\Delta t/3)}(x,v):= \sum_k^{n} \rt_k^{\:(t^*+\Delta t)}(x,v) \st_k^{\:(t^* + \Delta t/2)}(x,v), \nonumber\\
\rho^{\:(t^*+2\Delta t/3)}(x) := \int_{\Omega_v} f^{\:(t^*+2\Delta t/3)}(x,v)\ dv = \int_{\Omega_v} \sum_k^{n} \rt_k^{\:(t^*+\Delta t)} \st_k^{\:(t^* + \Delta t/2)} \ dv,\nonumber \\
-\Delta_x \varphi^{\:(t^*+2\Delta t/3)}(x) = 1 - \rho^{\:(t^*+2\Delta t/3)}(x), \\
E^{\:(t^*+2\Delta t/3)}(x) = -\nabla_x \varphi^{\:(t^*+2\Delta t/3)}(x). \nonumber
\end{eqnarray}

Defining 
\begin{eqnarray}
f^{\:(t^*)}(x,v):= \sum_k^{n} \rt_k^{\:(t^*)}(x,v) \st_k^{\:(t^*)}(x,v), \nonumber\\
f^{\:(t^*+\Delta t/3)}(x,v):= \sum_k^{n} \rt_k^{\:(t^*)}(x,v) \st_k^{\:(t^*+\Delta t/2)}(x,v), \nonumber\\
f^{\:(t^*+ 2\Delta t/3)}(x,v):= \sum_k^{n} \rt_k^{\:(t^*+\Delta t)}(x,v) \st_k^{\:(t^*+\Delta t/2)}(x,v), \nonumber\\
f^{\:(t^*+ \Delta t)}(x,v):= \sum_k^{n} \rt_k^{\:(t^*+\Delta t)}(x,v) \st_k^{\:(t^*+\Delta t)}(x,v), \nonumber\\
\end{eqnarray}
the above scheme can be rewritten as
\begin{eqnarray}
\left( I - \frac{\Delta t}{2} E^{\:(t^*)}\cdot \nabla_v \right) f^{\:(t^*+\Delta t/3)} = \left( I - \frac{\Delta t}{2} v\cdot \nabla_x \right) f^{\:(t^*)},\nonumber \\
\left( I + \frac{\Delta t}{2} v\cdot \nabla_x \right)f^{\:(t^*+ 2\Delta t/3)} =  f^{\:(t^*+\Delta t/3)}, \\
f^{\:(t^*+\Delta t)} = \left( I + \frac{\Delta t}{2} E^{\:(t^*+2\Delta t/3)}\cdot \nabla_v \right)f^{\:(t^*+ 2\Delta t/3)}. \nonumber
\end{eqnarray}

\medskip

This naturally leads us to define the following time-discretization scheme for the evolution of $f$. Set $f^{(0)}:= f_0$ and for all 
$m\in \N$, define
\begin{eqnarray}
\left( I - \frac{\Delta t}{2} E^{(m)}\cdot \nabla_v \right) f^{(m+1/3)} = \left( I - \frac{\Delta t}{2} v\cdot \nabla_x \right) f^{(m)},\nonumber \\
\left( I + \frac{\Delta t}{2} v\cdot \nabla_x \right)f^{(m+ 2/3)} =  f^{(m+1/3)}, \label{eq:scheme} \\ 
f^{(m+1)} = \left( I + \frac{\Delta t}{2} E^{(m+2/3)}\cdot \nabla_v \right)f^{(m+ 2/3)}. \nonumber \\ \nonumber \nonumber 
\end{eqnarray}

The function $f^{(m)}$ then gives an approximation of the solution $f$ to (\ref{eq:VlasovPoisson}) at time $t_m := m\Delta t$. 

\medskip

Of course, the starting point of the derivation of this scheme was to postulate that the function $f$ can be written in a separate variate expansion at some time $t^*$. 
In the next section, we present the algorithm which is used at each substep of the time-discretization scheme in order to obtain a tensorized approximation 
of the functions $f^{(m+1/3)}$, $f^{(m+2/3)}$ and $f^{(m+1)}$ for all $m\in\N$, assuming that $f^{(0)}$ is given in a tensorized form.

\section{Tensor methods}\label{sec:tensor}

In this section, let $H$, $H_x$ and $H_v$ be some arbitrary Hilbert spaces so that $H_x \otimes H_v \subseteq H$. A tensor-based method is introduced to solve the following problem: find $f\in H$ solution of
\begin{equation}\label{eq:genpbm}
(I + \Delta t P)f = g,
\end{equation}
where
\begin{itemize}
 \item $g$ is a finite-rank tensor product element of $H$; 
 \item $\Delta t\geq 0$ is a small constant;
 \item $I$ is an operator on $H$ of the form $I = I_x\otimes I_v$ where $I_x$ (respectively $I_v$) is a symmetric continous coercive operator on $H_x$ (respectively $H_v$);
 \item $P$ is an arbitrary \itshape tensorized \normalfont operator on $H$ (not necessarily symmetric).
\end{itemize}


In the Vlasov-Poisson context, each step of the proposed time-discretization scheme~(\ref{eq:scheme}) can be written 
under the form~(\ref{eq:genpbm}). Remark that the methodology presented hereafter can be directly applied to other time-discretization schemes and other contexts provided that 
they only require the resolution of elementary subproblems of the form (\ref{eq:genpbm}).  

\medskip

The approach relies on the so-called Proper Generalised Decomposition (PGD) method~\cite{ladeveze2010latin,chinesta2010,chinesta2011,falco2011,nouy2010priori}, and we first review
well-known results about this method in Section~\ref{sec:PGDsym}. We stress on 
the properties of this method on an important particular case in Section~\ref{sec:POD}. The scheme we propose is presented in Section~\ref{sec:PGDFP} 
along with convergence results whose proofs are postponed to the appendix.

\medskip

Let us highlight the philosophy of the method: the solution $f \in H$ of (\ref{eq:genpbm}) is approximated as a sum of tensor products
$$
f \approx \sum_{k=1}^n r_k \otimes s_k, 
$$
where for all $1\leq k \leq n$, $r_k \in H_x$ and $s_k \in H_v$. Each pair $(r_k, s_k)$ appearing in the above sum is computed in an iterative way so that the tensor product $r_k \otimes s_k$ is \itshape the best \normalfont possible 
tensor product at iteration $k$ of the algorithm. The meaning of this sentence will be made clear in the rest of the section. In the Vlasov-Poisson context, 
the advantage of this approach is that it only requires the resolution of linear problems for functions depending only on $x$ or only on $v$. Thus, the sizes of the linear 
problems that are solved are much smaller than the one of the full linear problems defining the iterations of the scheme (\ref{eq:scheme}). We comment further on this point in Section~\ref{sec:ALS}.

\subsection{PGD for coercive symmetric problems}\label{sec:PGDsym}

The PGD method is related to the so-called greedy algorithms~\cite{temlyakov2008greedy,le2009results} in nonlinear approximation theory. We review here well-known results on 
PGD algorithms for the approximation of high-dimensional coercive symmetric problems. We refer the reader to~\cite{falco2012, cances2011,figueroa2012} for more details. 

\medskip

Let $\bfa: H\times H \to \R$ a symmetric coercive continuous bilinear form on $H\times H$ and $\bfb:H\to \R$ a continuous linear form on $H$. 
Let $f\in H$ the unique solution of the linear problem
\begin{equation}\label{eq:sym}
\forall g\in H, \quad \bfa(f,g) = \bfb(g).
\end{equation}
The existence and uniqueness of a solution $f$ to problem (\ref{eq:sym}) is a consequence of the Lax-Milgram lemma. Besides, $f$ is equivalently the unique solution of the minimization problem
$$
f \in \mathop{\mbox{\rm argmin}}_{g\in H} \mathcal{E}(g), 
$$
were 
$$
\forall g \in H, \quad \mathcal{E}(g):= \frac{1}{2}\bfa(g,g) -  \bfb(g).
$$

Let us assume that the two Hilbert spaces $H_x$ and $H_v$  satisfy the following assumptions:
\begin{itemize}
 \item[(H1)] $\mbox{\rm Span}\left\{ r\otimes s, \; r\in H_x, \; s\in H_v \right\} \subset H$ and the inclusion is dense in $H$; 
 \item[(H2)] $\Sigma:=\left\{ r\otimes s, \; r\in H_x, \; s\in H_v \right\}$ is weakly closed in $H$.  
\end{itemize}

\medskip

Before presenting the PGD algorithm, we give here two simple examples of Hilbert spaces that satisfy these assumptions and are interesting in the Vlasov-Poisson context. 
\begin{enumerate}
 \item When $H = H^1_0(\Omega_x \times \Omega_v)$, the spaces $H_x = H^1_0(\Omega_x)$ and $H_v = H^1_0(\Omega_v)$ satisfy assumptions (H1)-(H2)~\cite{cances2011}.
 \item In a discretized setting, when $H = \R^{N_x \times N_v}$ for some $N_x,N_v\in \N^*$, the choice $H_x = \R^{N_x}$ and $H_v = \R^{N_v}$ ensures that (H1)-(H2) holds.  
\end{enumerate}

\vspace{1cm}

Let $g_0\in H$ be a given vector (which is usually chosen as $g_0 = 0)$. The PGD algorithm to compute an approximation of $f$ starting from the initial guess $g_0$ reads as follows:

\medskip

 \fbox{
\begin{minipage}{0.9\textwidth}
    {\bfseries PGD algorithm: }
 \begin{itemize}
  \item {\bfseries Initialization:} Set $n:=0$ and $f_0 := g_0$. 
  \item {\bfseries Iterate on $n\geq 0$:} Compute $(r_{n+1}, s_{n+1})\in H_x \times H_v$ as a solution of the minimization problem
  \begin{equation}\label{eq:PGDit}
  (r_{n+1}, s_{n+1}) \in \mathop{\mbox{argmin}}_{(r,s)\in H_x\times H_v} \mathcal{E}_n(r\otimes s),
  \end{equation}
  where for all $g\in H$, $\mathcal{E}_n(g) = \mathcal{E}(f_n + g) = \frac{1}{2}\bfa(f_n +g, f_n+g) - \bfb(f_n+g)$. 
  
  Define $f_{n+1} := f_n + r_{n+1} \otimes s_{n+1}$ and set $n = n+1$. 
 \end{itemize}
\end{minipage}
}

\vspace{1cm}

The choice of a stopping criterion is an important issue and we comment it later in this article. The method used in practice to solve (\ref{eq:PGDit}) is detailed in Section~\ref{sec:ALS}. 

\medskip

The following convergence result holds: 
\begin{prop}
Assume that the spaces $H,H_x,H_v$ satisfy assumptions (H1)-(H2). Then, all the iterations of the PGD algorithm are well-defined in the sense that there exists
at least one solution to problem~(\ref{eq:PGDit}) for all $n\in\N$. Besides, the sequence $(f_n)_{n\in\N}$ strongly converges 
in $H$ to $f$. 
\end{prop}

\medskip

We refer the reader to~\cite{le2009results,cances2011,falco2012,figueroa2012} for more details on the method and for the proof of this result. No further assumption is required at this stage on $\bfa$ 
or $\bfb$ for the convergence to hold. 
We will see in Section~\ref{sec:ALS} that the efficiency of a PGD-based method in practice depends on the tensor decomposition of $\bfa$ and $\bfb$. 

\subsection{An important particular case}\label{sec:POD}

A remarkable situation occurs when $H = H_x\otimes H_v$ and $\bfa(\cdot, \cdot) = \langle \cdot, \cdot \rangle_H$. 
In this case, it holds that 
\begin{equation}\label{eq:normproduct}
 \forall (r,s)\in H_x \times H_v, \quad \|r\otimes s\|_H = \|r\|_{H_x}\|s\|_{H_v}. 
\end{equation}

Let $\epsilon>0$ be a small positive constant which characterizes the stopping criterion. 

\medskip

\fbox{
\begin{minipage}{0.9\textwidth}
    {\bfseries PGD-$\epsilon$ algorithm: }
 \begin{itemize}
  \item {\bfseries Initialization:} Set $n=0$ and $f_0 = g_0$. 
  \item {\bfseries Iterate on $n\geq 0$:} Compute $(r_{n+1}, s_{n+1})\in H_x \times H_v$ as a solution of the minimization problem
  \begin{equation}\label{eq:PGDit2}
  (r_{n+1}, s_{n+1}) \in \mathop{\mbox{argmin}}_{(r,s)\in H_x\times H_v} \mathcal{E}_n(r\otimes s),
  \end{equation}
  where forall $g\in H$, $\mathcal{E}_n(g) = \frac{1}{2}\|f_n + g\|_H^2 - \bfb(g)$.
  Define $f_{n+1} := f_n + r_{n+1} \otimes s_{n+1}$. 
  
  If $\|r_{n+1} \otimes s_{n+1}\|_H < \epsilon$, then stop and define $PGD(b,g_0,\epsilon):= f_{n+1}$. Otherwise, set $n = n+1$ and iterate again. 
 \end{itemize}
\end{minipage}
}

\vspace{1cm}

Denoting by $b$ the Riesz representative of the linear form $\bfb$ for the scalar product $\langle \cdot, \cdot \rangle_H$, it holds that for all $n\in\N^*$, $(r_{n+1}, s_{n+1})$ is equivalently the solution of 
  \begin{equation}\label{eq:PGDit2bis}
  (r_{n+1}, s_{n+1}) \in \mathop{\mbox{argmin}}_{(r,s)\in H_x\times H_v} \| b - f_n - r\otimes s\|_H^2.
  \end{equation}
Using the norm-product property (\ref{eq:normproduct}), it can be proved~\cite{le2009results} that if $g_0=0$, the 
above algorithm gives an iterative method to compute the Proper 
Orthogonal Decomposition (POD) of the Riesz representative $b$ of the linear form $\bfb$ for the scalar product $\langle \cdot, \cdot \rangle_H$. A consequence is that 
an approximate solution $f_n = \sum_{k=1}^n r_k \otimes s_k$ computed after $n$ iterations of the PGD algorithm 
is a best $n$-rank approximation of $b$. In other words, 
$$
\| b - f_n \|_H = \min_{ (\widetilde{r}_k,\widetilde{s}_k)_{1\leq k \leq n } \in (H_x \times H_v)^n} \left\| b - \sum_{k=1}^n \widetilde{r}_k \otimes \widetilde{s}_k\right\|_H.
$$
This optimality property is particularly interesting in the present case. In the rest of the article, we shall denote by $POD(b, \epsilon) = PGD(b ,0, \epsilon)$.

\medskip

Another interesting consequence is that the sequence $\left( \|r_n\otimes s_n\|_H \right)_{n\in \N^*}$ of the norms of the tensor product functions given by the PGD algorithm is non-increasing. Indeed, 
this sequence is identical to the set of the singular values of the POD of $b$ in $H_x\otimes H_v$ in non-increasing order~\cite{le2009results}.  

\medskip

Let us comment here on the use of this particular stopping criterion, which is the one we use in practice in the Vlasov-Poisson context. It holds from (\ref{eq:PGDit2}) (or equivalently (\ref{eq:PGDit2bis})) that 
$$
\|r_{n+1}\otimes s_{n+1}\|_H = \max_{(r,s)\in H_x \times H_v} \frac{\langle b-f_n, r\otimes s\rangle_H}{\|r\otimes s\|_H}.
$$
For any element $g\in H$, let us define
$$
\|g\|_*:= \mathop{\sup}_{(r,s)\ni H_x \times H_v} \frac{\langle g, r\otimes s\rangle_H}{\|r\otimes s\|_H} =  \mathop{\max}_{(r,s)\ni H_x \times H_v} \frac{\langle g, r\otimes s\rangle_H}{\|r\otimes s\|_H}  .
$$

The application $g\in H \mapsto \|g\|_*\in \R_+$ defines a norm on $H$ which is called the injective norm~\cite{grothendieck1996resume}. This norm is equal to the maximal singular value of the POD decomposition of 
$g$ in $H_x\otimes H_v=H$. Of course, we have $\|g\|_* \leq \|g\|_H$ but these two norms are not equivalent. Thus, $\|r_{n+1}\otimes s_{n+1}\|_H$ can be seen as the injective norm of the residual of the decomposition 
$b-f_n$. We use the injective norm in practice as a stopping criterion because the latter 
quantity is much faster to evaluate than the $H$-norm. 

\subsection{Fixed-point PGD algorithm for weakly non-symmetric problems}\label{sec:PGDFP}

Assume now that $f$ is the solution of a problem of the form
\begin{equation}\label{eq:nonsym}
\forall g\in H, \quad \langle f,g \rangle_H + \widetilde{\bfa}(f,g) = \bfb(g),
\end{equation}
where $\bfb$ is a continuous linear form on $H$ and $\widetilde{\bfa}: H\times H \to \R$ is a continuous bilinear form which is not symmetric nor coercive in general. There exists a unique solution of this problem 
for instance when $\| \widetilde{\bfa}\|_{\cL(H\times H; \R)} < 1$.

\medskip

We still assume that we start from an initial guess for $f$ given by an element $g_0 \in H$. A natural idea to solve (\ref{eq:nonsym}) when $\widetilde{\bfa}$ is 
a \itshape small \normalfont perturbation of the identity operator on $H$ is to consider the following fixed-point PGD algorithm: 

\medskip

\fbox{
\begin{minipage}{0.9\textwidth}
    {\bfseries Fixed-point PGD algorithm: }
 \begin{itemize}
  \item {\bfseries Initialization:} Set $n=0$ and $f_0 = g_0$. 
  \item {\bfseries Iterate on $n\geq 0$:} Compute $(r_{n+1}, s_{n+1})\in H_x \times H_v$ as a solution of the minimization problem
  \begin{equation}\label{eq:PGDit3}
  (r_{n+1}, s_{n+1}) \in \mathop{\mbox{argmin}}_{(r,s)\in H_x\times H_v} \mathcal{E}_n(r\otimes s),
  \end{equation}
  where for all $g\in H$, $\mathcal{E}_n(g) = \frac{1}{2}\langle f_n +g, f_n+g \rangle_H - \bfb(f_n+g) - \widetilde{\bfa}(f_n, f_n+g)$. 
  
  Define $f_{n+1} := f_n + r_{n+1} \otimes s_{n+1}$ and set $n = n+1$. 
 \end{itemize}

\end{minipage}
}

\vspace{1cm}

This algorithm was already suggested and studied in~\cite{cances2013}. 
Its convergence was then proved under the condition that the Hilbert spaces $H_x$ and $H_v$ are finite-dimensional and that $\|\widetilde{\bfa}\|_{\cL(H\times H)} \leq \kappa$ where $\kappa$ was 
some constant depending on the dimension of the spaces which goes to $0$ as the dimension of the spaces go to infinity. This theoretical convergence result was much more pessimistic than the numerical observations. 
Indeed, it was already pointed out in~\cite{cances2013} that numerical tests indicated that this constant $\kappa$ should not depend on the dimension of the spaces. 

\medskip

In this article, we prove that $\kappa$ does not need to depend on the dimension of the Hilbert spaces, but on the number of terms appearing in the tensor decomposition of $\widetilde{\bfa}$.
More precisely, let $\widetilde{A} \in \cL(H; H)$ be the continuous linear operator on $H$ associated to $\widetilde{\bfa}$, i.e. such that 
$$
\forall g_1,g_2\in H, \quad \widetilde{\bfa}(g_1, g_2) = \langle \widetilde{A} g_1, g_2 \rangle_H.
$$
Then, the following result holds: 

\begin{prop}\label{prop:FPPGD}
All the iterations of the Fixed-point PGD algorithm are well-defined, in the sense that for all $n\in\N$, there exists at least one solution to (\ref{eq:PGDit}). 
Moreover, let us assume that $\widetilde{A} = \sum_{\mu=1}^M \widetilde{A}_x^\mu \otimes \widetilde{A}_v^\mu$ where for all $1\leq \mu \leq M$, $\widetilde{A}_x^\mu \in \cL(H_x;H_x)$ and $\widetilde{A}_v^\mu \in \cL(H_v;H_v)$. 
 Let $\kappa:= \max_{1\leq \mu \leq M} \left\| \widetilde{A}_x^\mu \otimes \widetilde{A}_v^\mu\right\|_{\cL(H; H)}$. Assume that at least one of these two assumptions is satisfied: 
 \begin{itemize}
  \item [(A1)] $H = H_x\otimes H_v$ (thus the norm of $H$ satisfies the norm-product property (\ref{eq:normproduct})) and $3M\kappa < 1$; 
  \item [(A2)] $5M\kappa < 1$. 
 \end{itemize}
Then, there is a unique solution $f$ to (\ref{eq:nonsym}) and the sequence $(f_n)_{n\in\N}$ strongly converges in $H$ to~$f$.
\end{prop}
The proof of Proposition~\ref{prop:FPPGD} is given in the appendix. Let us point out that the convergence of the proposed algorithm is not covered in the work~\cite{hackbusch2008}, where 
the authors also treat approximation of equations using tensor methods and fixed-point iterations, but 
with a different point of view.

\medskip

In the Vlasov-Poisson context, a similar stopping criterion is used, as the one we described in Section~\ref{sec:POD}. More precisely, for $\epsilon>0$, $g_0\in H$, we consider the following algorithm: 

\medskip

\fbox{
\begin{minipage}{0.9\textwidth}
    {\bfseries Fixed-point PGD-$\epsilon$ algorithm: }
 \begin{itemize}
  \item {\bfseries Initialization:} Set $n=0$ and $f_0 = g_0$. 
  \item {\bfseries Iterate on $n\geq 0$:} Compute $(r_{n+1}, s_{n+1})\in H_x \times H_v$ as a solution of the minimization problem
  \begin{equation}\label{eq:PGDit3bis}
  (r_{n+1}, s_{n+1}) \in \mathop{\mbox{argmin}}_{(r,s)\in H_x\times H_v} \mathcal{E}_n(r\otimes s),
  \end{equation}
  where for all $g\in H$, $\mathcal{E}_n(g) = \frac{1}{2}\langle f_n +g, f_n+g \rangle_H - \bfb(f_n+g) - \widetilde{\bfa}(f_n, f_n+g)$. 
  
  Define $f_{n+1} := f_n + r_{n+1} \otimes s_{n+1}$.

  If $\|r_{n+1} \otimes s_{n+1}\|_H < \epsilon$, then stop and define $PGD_{FP}(\widetilde{A}, b,g_0,\epsilon):= f_{n+1}$. Otherwise, set $n = n+1$ and iterate again. 
 \end{itemize}
\end{minipage}
}

\vspace{1cm}

Let $b\in H$ denote the Riesz representative of $\bfb$ in $H$. For all $n\in\N$, $(r_{n+1},s_{n+1})\in H_x \times H_v$ is a solution to (\ref{eq:PGDit3}) if and only if it is a solution to 
\begin{equation}\label{eq:alter}
(r_{n+1}, s_{n+1}) \in \mathop{\mbox{argmin}}_{(r,s)\in H_x \times H_v} \| b - (I + \widetilde{A})f_n - r\otimes s\|_H^2,
\end{equation}
where $I$ denotes the identity operator on $H$. The stopping criterion used above is justified by the fact that, for all $n\in\N$, $\|r_{n+1}\otimes s_{n+1}\|_H$ 
is equal to the injective norm of the residual of the equation $R_n:= b - (I + \widetilde{A})f_n$. Indeed, (\ref{eq:alter}) implies that 
$$
\|r_{n+1}\otimes s_{n+1}\|_H = \mathop{\max}_{(r,s)\in H_x \times H_v} \frac{\langle R_n, r\otimes s\rangle_H}{\|r\otimes s\|_H} = \|R_n\|_*.
$$

\subsection{Alternating least squares (ALS) for the practical resolution of the PGD iterations}~\label{sec:ALS}

We present in this section how minimization problems (\ref{eq:PGDit}), (\ref{eq:PGDit2}), (\ref{eq:PGDit3}) and (\ref{eq:PGDit3bis}) are solved in practice. Let us point out that in all cases, at iteration $n\in\N$, $(r_{n+1},s_{n+1})\in H_x \times H_v$ 
is defined as a solution to 
\begin{equation}\label{eq:min2solve}
(r_{n+1}, s_{n+1}) \in \mathop{\mbox{argmin}}_{(r,s)\in H_x \times H_v} \mathcal{E}_n(r\otimes s), 
\end{equation}
where for all $g\in H$, $\mathcal{E}_n(g) = \frac{1}{2}\bfc(g,g) - \bfl(g)$, for some continuous linear form $\bfl: H \to \R$ and coercive bilinear continuous form $\bfc: H \times H \to \R$, wich depend on $n$.

\medskip

Problem (\ref{eq:min2solve}) is solved in practice using the Alternating Least Squares (ALS)~\cite{espig2015, uschmajew2012local,rohwedder2013local} algorithm which is standard in tensor-based approximation methods. 
For a given error tolerance $\eta>0$, the algorithm reads as follows:

\medskip

\fbox{
\begin{minipage}{0.9\textwidth}
    {\bfseries ALS-$\epsilon$ algorithm: }
 \begin{itemize}
  \item {\bfseries Initialization:} Set $m=0$ and choose randomly $r^0\in H_x$ and $s^0 \in H_v$. 
  \item {\bfseries Iterate on $m\geq 0$:} Compute $r^{m+1}\in H_x$ as the unique solution of
  \begin{equation}\label{eq:rm}
  r^{m+1}\in \mathop{\mbox{argmin}}_{r\in H_x} \mathcal{E}_n(r\otimes s^m).
  \end{equation}
  Then, compute $s^{m+1}\in H_v$ as the unique solution of 
  \begin{equation}\label{eq:sm}
  s^{m+1} \in \mathop{\mbox{argmin}}_{(s\in H_v} \mathcal{E}_n(r^{m+1}\otimes s).
  \end{equation}

  If $\|r^{m+1}\otimes s^{m+1} - r^m\otimes s^m\|_H < \eta$, set $r_{n+1} = r^{m+1}$ and $s_{n+1} = s^{m+1}$. 
  
  Otherwise, set $m:= m+1$ and iterate again. 
 \end{itemize}
\end{minipage}
}

\medskip

The convergence properties of this ALS algorithm are analyzed in details in~\cite{espig2015} in the case when $H_x$ and $H_v$ are finite-dimensional, and for more sophisticated tensor formats. The algorithm can be shown to converge to a solution 
of the Euler equations associated to (\ref{eq:min2solve}). The limit tensor product is not theoretically ensured to be the global minimum (or even a local minimum) of $(r,s)\in H_x \times H_v \mapsto \mathcal{E}_n(r\otimes s)$. 

However, in practice, one can observe that it usually converges in a few iterations to a local minimum of (\ref{eq:min2solve}). 
It is very commonly observed that this choice leads to very satisfactory convergence rates of PGD methods. Hence, we also use it here in the Vlasov-Poisson context.

\medskip

In the rest of the article, we shall denote by $PGD(b,g_0,\epsilon,\eta)$ (respectively $PGD_{FP}(\widetilde{A}, b, g_0, \epsilon, \eta)$) the functions obtained by 
the PGD-$\epsilon$ (respectively Fixed-point PGD-$\epsilon$) method when 
minimization problem (\ref{eq:PGDit2bis}) (respectively (\ref{eq:PGDit3bis})) is solved using an ALS-$\eta$ algorithm. We also denote by $POD(b, \epsilon, \eta):= PGD(b, 0, \epsilon, \eta)$.

\medskip

We stress here on a crucial point: for the ALS algorithm to be numerically efficient in a high-dimensional context, it is important that the forms $\bfc$ and $\bfl$ 
admits a finite-rank tensor decomposition. Indeed, let us assume that $\bfc = \sum_{\gamma = 1}^C \bfc_x^\gamma \otimes \bfc^\gamma_v$ and $\bfl = \sum_{\delta = 1}^D \bfl_x^\delta \otimes \bfl_y^\delta$ for some $C,D\in \N^*$, 
such that for all $1\leq \gamma \leq C$, $\bfa_x^\gamma \in \cL(H_x \times H_x; \R)$, $\bfa_v^\gamma \in \cL(H_v\times H_v; \R)$, and for all $1\leq \delta \leq D$, $\bfl_x^\delta \in \cL(H_x; \R)$, $\bfl_v^\delta \in \cL(H_v; \R)$. 

\medskip

At each iteration $m\in\N^*$ of the ALS-$\eta$ algorithm, $r^{m+1}\in H_x$ (respectively $s^{m+1}\in H_v$) is the unique solution to (\ref{eq:rm}) (respectively (\ref{eq:sm})) if and only if it is the solution of 
the first-order Euler equations
\begin{eqnarray*}
\forall r\in H_x, \quad  \bfc(r^{m+1}\otimes s^m , r\otimes s^m) = \bfl(r\otimes s^m),\\
 \forall s\in H_v, \quad  \bfc(r^{m+1}\otimes s^{m+1} , r^{m+1}\otimes s) = \bfl(r^{m+1}\otimes s).\\ 
\end{eqnarray*}
We clearly see that the computation of $r^{m+1}$ and $s^{m+1}$ only requires the resolution of a linear symmetric coercive system for functions depending only on $x$, or only on $v$. 
The size of the associated discretized 
problems are thus much smaller than those that one would have obtained to solve (\ref{eq:nonsym}) directly for instance. 

Using the tensor decomposition of $\bfc$ and $\bfl$, these equations can be rewritten as
\begin{eqnarray}
\forall r\in H_x, \quad  \sum_{\gamma=1}^C \bfc^\gamma_v(s^m , s^m)\bfc_x^\gamma(r^{m+1}, r) = \sum_{\delta=1}^D\bfl_v^\delta(s^m)\bfl_x^\delta(r),\\ \label{eq:eqr}
 \forall s\in H_v, \quad \sum_{\gamma=1}^C \bfc^\gamma_v(s^{m+1} , s)\bfc_x^\gamma(r^{m+1}, r^{m+1}) = \sum_{\delta=1}^D\bfl_v^\delta(s)\bfl_x^\delta(r^{m+1}).\\ \label{eq:eqs} \nonumber
\end{eqnarray}
The tensorized decomposition of $\bfc$ and $\bfl$ implies that each term appearing in (\ref{eq:eqr}) and (\ref{eq:eqs}) can be quickly evaluated, since they only involve forms defined on Hilbert spaces of functions depending 
on only one variable. Such tensorized decompositions are always naturally available in the Vlasov-Poisson context, and this crucial fact is at the heart of the efficiency of this approach.

\section{Final algorithm for the Vlasov-Poisson system}

\subsection{Space and velocity discretization}\label{sec:vxdisc}

Let us highlight that the proposed method can be adapted to various types of space and velocity discretizations, as well as different time schemes. 
It can also be adapted in other contexts than the Vlasov system.
Let us assume that a discretization with $N_x$ (respectively $N_v$) degrees of freedom in the $x$ variable (respectively the $v$ variable) is used. Thus, at each time step of the 
discretization scheme, the approximation of $f$ is characterized 
by a matrix $f \in \R^{N_x\times N_v}$ which is computed in a separated form as
$$
f  = \sum_{k=1}^n r_k\otimes s_k = \sum_{k=1}^n r_k s_k^T, 
$$
with some vectors $r_k \in \R^{N_x\times 1}$ and $s_k \in \R^{N_v\times 1}$. 

\medskip

In this setting, the Hilbert spaces $H_x$, $H_v$ and $H$ are chosen to be $\R^{N_x}$, $\R^{N_v}$ and $\R^{N_x \times N_v}$. The only requirement for this strategy to be applicable is that 
each step of the chosen scheme requires the resolution of problems of the form (\ref{eq:genpbm}), where $P$ is a tensorized operator at the discrete level, and $g$ a finite-rank element of $H$. More precisely, 
we assume that at each step of the algorithm $P$ and $g$ can be respectively written as 
$$
P = \sum_{k=1}^p P_x^k \otimes P_v^k \mbox{ and } g = \sum_{k=1}^q g_x^k \otimes g_v^k,
$$
for some matrices $P_x^k \in \R^{N_x \times N_x}$, $P_v^k \in \R^{N_v \times N_v}$ and vectors $g_x^k\in\R^{N_x}$, $g_v^k\in\R^{N_v}$. 
Also, the operator $I$ appearing in equation~(\ref{eq:genpbm}) can be written as $I = I_x \otimes I_v$ for some symmetric positive matrices $I_x\in \R^{N_x \times N_x}$ and $I_v \in \R^{N_v \times N_v}$. 
The space $H$ is endowed with the scalar product
$$
\forall f,g\in H, \quad \langle f,g \rangle_H:= \mbox{\rm Tr}\left( f^T I_x g I_v \right) = \mbox{\rm Tr}\left( f^T (  I_x \otimes I_v g) \right).
$$

\medskip

The idea is illustrated using a finite element discretization. Let us introduce $\left( \phi_i(x) \right)_{1\leq i \leq N_x}$ and $\left( \psi_k(v)\right)_{1\leq k \leq N_v}$ some finite element 
discretization bases of functions defined on $\Omega_x$ and $\Omega_v$ respectively. Assume that these functions belong respectively to $H^1(\Omega_x)$ and $H^1(\Omega_v)$ with appropriate 
boundary conditions. 

The following matrices are defined: for all $1\leq \alpha \leq d$ (recall that $d$ is the dimension of the problem), and any measurable bounded field 
$E = (E_{\alpha})_{1\leq \alpha \leq d}: \Omega_x \to \R^d$, 
\begin{eqnarray*}
I_x &:= \left( \int_{\Omega_x} \phi_i(x) \phi_j(x)\,dx \right)_{1\leq i,j \leq N_x}, \\
F_x (E_\alpha) & := \left( \int_{\Omega_x}  \phi_i(x) E_\alpha(x) \phi_j(x)\,dx \right)_{1\leq i,j \leq N_x}, \\
D_{\alpha,x} & := \left( \int_{\Omega_x}  \phi_i(x) \partial_{x_\alpha} \phi_j(x)\,dx \right)_{1\leq i,j \leq N_x}, \\
I_v & := \left( \int_{\Omega_v} \psi_k(v) \psi_l(v)\,dv \right)_{1\leq k,l \leq N_v}, \\
V_{\alpha,v}& := \left( \int_{\Omega_v}  \psi_k(v) v_\alpha \psi_l(v) \,dv \right)_{1\leq k,l \leq N_v}, \\
D_{\alpha,v} & := \left( \int_{\Omega_v}  \psi_k(v) \partial_{v_\alpha} \psi_l(v)\,dv \right)_{1\leq k,l \leq N_v}.\\
\end{eqnarray*}

The time scheme introduced in Section~\ref{sec:scheme} is recalled:
\begin{eqnarray}
\left( I + \frac{\Delta t}{2} E^{(m)}\cdot \nabla_v \right) f^{(m+1/3)} = \left( I - \frac{\Delta t}{2} v\cdot \nabla_x \right) f^{(m)},\nonumber \\
\left( I + \frac{\Delta t}{2} v\cdot \nabla_x \right)f^{(m+ 2/3)} =  f^{(m+1/3)}, \\
f^{(m+1)} = \left( I - \frac{\Delta t}{2} E^{(m+2/3)}\cdot \nabla_v \right)f^{(m+ 2/3)}. \nonumber
\end{eqnarray}

The discretized version of this scheme then reads as follows: let $f^{(0)} \in \R^{N_x \times N_v}$. For all $m\in \N$, compute $f^{(m+1/3)},f^{(m+2/3)},f^{(m+1)}\in \R^{N_x \times N_v}$ solutions of
\begin{eqnarray}
\left( I_x \otimes I_v + \frac{\Delta t}{2} \sum_{\alpha=1}^d F_x(E_{\alpha}^{(m)})\otimes D_{\alpha,v} \right) f^{(m+1/3)} = \left( I_x \otimes I_v - \frac{\Delta t}{2} V_{\alpha,v} \otimes D_{\alpha,x} \right) f^{(m)},\nonumber \\
\left( I_x \otimes I_v + \frac{\Delta t}{2} \sum_{\alpha=1}^d D_{\alpha, x} \otimes V_{\alpha, v} \right)f^{(m+ 2/3)} =  f^{(m+1/3)}, \\
f^{(m+1)} = \left( I_x \otimes I_v - \frac{\Delta t}{2}\sum_{\alpha=1}^d F_x(E_{\alpha}^{(m+2/3)})\otimes D_{\alpha,v} \right)f^{(m+ 2/3)}. \nonumber
\end{eqnarray}

\subsection{Summary of the algorithm in the discretized setting}

The method we propose for the resolution of the Vlasov-Poisson system is summarized hereafter. Let $\epsilon>0$ be a chosen tolerance threshold.

\medskip

\fbox{
\begin{minipage}{0.9\textwidth}
    {\bfseries Verlet-PGD-$\epsilon$ algorithm: }
 \begin{itemize}
  \item {\bfseries Initialization:} Set $f^{(0)} = f_0$. 
  \item {\bfseries Iterate on $m\geq 0$:} 
  \begin{itemize}
   \item 
   Define $P^{(m+1/3)} = \frac{\Delta t}{2} \sum_{\alpha=1}^d F_x(E_\alpha^{(m)})\otimes D_{\alpha,v}$ and $g^{(m+1/3)} = \left( I_x \otimes I_v - \frac{\Delta t}{2} V_{\alpha,v} \otimes D_{\alpha,x} \right) f^{(m)}$. 
   Compute $\overline{f}^{(m+1/3)}$ as 
   $$
   \overline{f}^{(m+1/3)} = PGD_{FP}(P^{(m+1/3)},g^{(m+1/3)}, f^{(m)}, \epsilon, \epsilon).
   $$
   Recompress $\overline{f}^{(m+1/3)}$ by computing 
   $$
   f^{(m+1/3)} = POD\left(\overline{f}^{(m+1/3)}, \epsilon, \epsilon\right). 
   $$
   \item Define $P^{(m+2/3)} =  \frac{\Delta t}{2} \sum_{\alpha=1}^d D_{\alpha, x} \otimes V_{\alpha, v} $. 
   Compute $\overline{f}^{(m+2/3)}$ as 
   $$
   \overline{f}^{(m+2/3)} = PGD_{FP}(P^{(m+1/3)},f^{(m+1/3)}, f^{(m+1/3)}, \epsilon, \epsilon).
   $$
   Recompress $\overline{f}^{(m+2/3)}$ by computing 
   $$
   f^{(m+2/3)} = POD\left(\overline{f}^{(m+2/3)}, \epsilon, \epsilon\right). 
   $$
   \item Define $Q^{(m+1)} = - \frac{\Delta t}{2}\sum_{\alpha=1}^d F_x(E_{\alpha}^{(m+2/3)})\otimes D_{\alpha,v}$. Compute $f^{(m+1)}$ as 
   $$
   f^{(m+1)} = POD((I_x\otimes I_v + Q^{(m+1)})f^{(m+2/3)}, \epsilon, \epsilon).
   $$
  \end{itemize}
 \end{itemize}
\end{minipage}

}

\vspace{1cm}

The condition we obtained in Proposition~\ref{prop:FPPGD} on the convergence of the Fixed-point PGD algorithm implies that the time step $\Delta t>0$ has to be taken sufficiently small 
to ensure that the norms of the operators entering in the decomposition of $P^{m+1/3}$ and $P^{m+2/3}$ are also small. In practice, we thus observe that our scheme suffers from a type of CFL condition that 
has to be respected for the method to converge. Apart from this restriction which does not appear to be too penalizing in practice, the approach proposed here is very flexible and yields promising numerical results 
as shown in the next section.

\section{Numerical results}~\label{sec:numtest}
In this section some numerical experiments are presented, to assess the properties of the method. 
First, two 1D-1D examples are considered, to validate the proposed approach. The following quantities are monitored: the error in mass, momentum and energy conservation, 
and the error with respect to a reference solution. The $L_2$ averaged in time relative errors are defined as follows:
\begin{eqnarray}
\epsilon_m := \frac{1}{M\, t_f } \left( \int_0^{t_f} ( m - m(0) )^2 \ dt \right)^{1/2}, \\
\epsilon_p := \frac{1}{P\, t_f} \left(  \int_0^{t_f} (p - p(0) )^2 \ dt \right)^{1/2}, \\
\epsilon_h := \frac{1}{H(0)\, t_f} \left( \int_0^{t_f} (h-h(0))^2) \ dt \right)^{1/2}, \\
\epsilon_f := \frac{1}{ t_f} \left( \int_0^{t_f} \frac{\int_{\Omega}  (f_{ref}-f)^2 \ dx \ dv}{ \int_{\Omega}  f_{ref}^2 \ dx \ dv}  \ dt \right)^{1/2},  
\end{eqnarray}
where $t_f$ is the final time of the simulation, $M$ is the normalising mass factor, defined as the mass of the initial condition 
$M = m(0) = \int_{\Omega} f_0 \ dx \ dv $, $P = \sqrt{2 M K}$ is the momentum reference value, where $K = \int_{\Omega} f_0 \frac{v^2}{2} \ dx \ dv $ is the initial 
kinetic energy, and $H(0)$ is the Hamiltonian at initial time.

In the last part of this section, a 2D-2D example is shown to illustrate the applicability of the method in more high-dimensional settings. Simulations on 3D-3D testcases is work in progress. 

\subsection{Landau Damping.}
\label{subsec:LandauDamping}
The first test proposed is a standard linear Landau damping in a 1D-1D configuration, as proposed in~\cite{madaule2014}.
The domain size is $\Omega_x = [0,4\pi]$ and $\Omega_v = [-10,10]$. Periodic (respectively homogeneous Dirichlet) boundary conditions are set on $\Omega_x$ (respectively $\Omega_v$).
The initial condition is given in analytical form as:
\begin{eqnarray}
f(x,v; t=0) = F(x)G(v), \\
F(x) = 1 + \beta \cos(k x), \\
G(v) = \frac{1}{\sqrt{2\pi}} \exp \left( - \frac{v^2}{2} \right),
\end{eqnarray}
where $k = 0.5$ is the wavenumber of the perturbation and the amplitude $\beta=0.01$ set the problem in a linear Landau damping regime (see~\cite{madaule2014,kormann2015}). In such a 
configuration the analytical decay rate for the electric amplitude is $\gamma \approx 0.153$. For this test a mixed discretization is set up: 
for the space, a spectral collocation method is used based on a Fourier discretization, whereas for the velocity standard centered finite differences are used.

The numerical experiments are done by varying the space and velocity resolution, the time step, and the tolerance on the residual. For the space and the velocity discretization, we take 
$N_x=N_v = (32,64,128,256)$. The final time is set to $t_f=10.0$ and the number of iterations is $N_t = (4 \cdot 10^3, 8 \cdot 10^3, 16 \cdot 10^3)$. The tolerance on the 
residual is chosen as $\varepsilon = (10^{-10}, 10^{-12},  10^{-14}, 10^{-16})$. For the reference simulation $N_x=N_v = 512$, $N_t = 32 \cdot 10^3$ and $\varepsilon = 10^{-18}$. 

\begin{table}
\caption{Errors in the conserved quantities and with respect to a reference simulation for the 1D-1D Landau Damping testcase (section \ref{subsec:LandauDamping} )}
\begin{tabular}{c c c c c}
\hline
resolution ($N_x$ -- $N_t$ -- $\varepsilon$) & $\epsilon_m$ & $\epsilon_p$ & $\epsilon_h$ & $\epsilon_f$ \\
\hline
32 -- $4\cdot 10^3$ -- $10^{-10}$ & $ 9.09 \cdot 10^{-7}$ & $5.18\cdot 10^{-6} $ & $8.52 \cdot 10^{-5} $  & $1.06 \cdot 10^{-3} $ \\
32 -- $4\cdot 10^3$ -- $10^{-12}$ & $1.12 \cdot 10^{-6} $ & $6.24\cdot 10^{-6} $ & $2.43\cdot 10^{-5} $  & $4.15\cdot10^{-4} $ \\
32 -- $4\cdot 10^3$ -- $10^{-14}$ & $1.02\cdot 10^{-7} $ & $5.59 \cdot 10^{-6} $ & $1.01 \cdot 10^{-5} $  & $4.14 \cdot 10^{-4} $\\
32 -- $4\cdot 10^3$ -- $10^{-16}$ & $7.08 \cdot 10^{-8} $ & $5.60 \cdot 10^{-6}$ & $9.45 \cdot 10^{-6} $  & $4.14 \cdot{10^{-4}}$\\
\hline
32 -- $8\cdot 10^3$ -- $10^{-10}$ & $3.64 \cdot 10^{-6}$ & $6.90 \cdot 10^{-6}$ & $9.31 \cdot 10^{-4}$  & $3.28 \cdot 10^{-3}$\\
32 -- $8\cdot 10^3$ -- $10^{-12}$ & $1.11 \cdot 10^{-6}$ & $6.28 \cdot 10^{-6}$ & $2.48 \cdot 10^{-5}$  & $4.05 \cdot 10^{-4}$\\
32 -- $8\cdot 10^3$ -- $10^{-14}$ & $4.32 \cdot 10^{-7}$ & $5.60 \cdot 10^{-6}$ & $1.60 \cdot 10^{-5}$  & $4.05 \cdot 10^{-4}$\\
32 -- $8\cdot 10^3$ -- $10^{-16}$ & $6.71 \cdot 10^{-8}$ & $5.64 \cdot 10^{-6}$ & $8.75 \cdot 10^{-6}$  & $4.05 \cdot 10^{-4} $\\
\hline
32 -- $16\cdot 10^3$ -- $10^{-10}$ & $4.38 \cdot 10^{-6}$ & $7.57 \cdot 10^{-6}$ & $1.04 \cdot 10^{-3}$  & $3.40 \cdot 10^{-3}$\\
32 -- $16\cdot 10^3$ -- $10^{-12}$ & $1.04 \cdot 10^{-6}$ & $6.31 \cdot 10^{-6} $ & $2.32\cdot 10^{-5}$  & $4.01\cdot 10^{-4}$\\
32 -- $16\cdot 10^3$ -- $10^{-14}$ & $1.14 \cdot 10^{-6}$ & $6.31 \cdot 10^{-6}$ & $2.46\cdot 10^{-5}$  & $4.01 \cdot 10^{-4}$ \\
32 -- $16\cdot 10^3$ -- $10^{-16}$ & $6.01\cdot 10^{-8}$ & $5.67 \cdot 10^{-6}$ & $8.53 \cdot 10^{-6}$  & $4.00 \cdot 10^{-4}$ \\
\hline
\hline
64 -- $4\cdot 10^3$ -- $10^{-10}$ & $6.75 \cdot 10^{-7}$ & $2.61 \cdot 10^{-6}$ & $8.33 \cdot 10^{-5}$  & $9.71 \cdot 10^{-4}$ \\
64 -- $4\cdot 10^3$ -- $10^{-12}$ & $8.58 \cdot 10^{-7}$ & $3.19 \cdot 10^{-6}$ & $1.98\cdot 10^{-5}$  & $1.33 \cdot 10^{-4}$ \\
64 -- $4\cdot 10^3$ -- $10^{-14}$ & $2.11 \cdot 10^{-7}$ & $2.68 \cdot 10^{-6}$ & $1.08 \cdot 10^{-5}$  & $1.28 \cdot 10^{-4}$ \\
64 -- $4\cdot 10^3$ -- $10^{-16}$ & $1.54 \cdot 10^{-8}$ & $2.68 \cdot 10^{-6}$ & $8.32 \cdot 10^{-6}$  & $1.28 \cdot 10^{-4}$ \\
\hline
64 -- $8\cdot 10^3$ -- $10^{-10}$ & $1.21 \cdot 10^{-6}$ & $ 2.70 \cdot 10^{-6} $ & $3.29 \cdot 10^{-4}$  & $2.91 \cdot 10^{-3}$ \\
64 -- $8\cdot 10^3$ -- $10^{-12}$ & $8.70 \cdot 10^{-7}$ & $3.19 \cdot 10^{-6}$ & $1.92 \cdot 10^{-5}$  & $1.30 \cdot 10^{-4}$ \\
64 -- $8\cdot 10^3$ -- $10^{-14}$ & $8.64 \cdot 10^{-7}$ & $2.61\cdot 10^{-6}$ & $2.14\cdot 10^{-5}$  & $1.28 \cdot 10^{-4}$ \\
64 -- $8\cdot 10^3$ -- $10^{-16}$ & $1.77 \cdot 10^{-8}$ & $2.69 \cdot 10^{-6}$ & $8.38 \cdot 10^{-6}$  & $1.25 \cdot 10^{-4}$ \\
\hline
64 -- $16\cdot 10^3$ -- $10^{-10}$ & $1.42 \cdot 10^{-6} $ & $3.61 \cdot 10^{-6}$ & $2.62 \cdot 10^{-4}$  & $3.30 \cdot 10^{-3}$ \\
64 -- $16\cdot 10^3$ -- $10^{-12}$ & $9.95 \cdot 10^{-7}$ & $3.19 \cdot 10^{-6}$ & $2.15 \cdot 10^{-5}$  & $1.28 \cdot 10^{-4}$ \\
64 -- $16\cdot 10^3$ -- $10^{-14}$ & $9.25 \cdot 10^{-7}$ & $3.19 \cdot 10^{-6}$ & $2.03 \cdot 10^{-5}$  & $1.28 \cdot 10^{-4}$ \\
64 -- $16\cdot 10^3$ -- $10^{-16}$ & $3.10 \cdot 10^{-8}$ & $2.71 \cdot 10^{-6}$ & $8.54 \cdot 10^{-6}$  & $1.23 \cdot 10^{-4}$ \\
\hline
\hline
\end{tabular}
\label{tab:LandauDampingErrors_1}
\end{table}

\begin{table}
\begin{tabular}{c c c c c}
\hline
resolution ($N_x$ -- $N_t$ -- $\varepsilon$) & $\epsilon_m$ & $\epsilon_p$ & $\epsilon_h$ & $\epsilon_f$ \\
\hline
128 -- $4\cdot 10^3$ -- $10^{-10}$ & $4.45 \cdot 10^{-7}$ & $1.29 \cdot 10^{-6}$ & $9.31 \cdot 10^{-5}$  & $1.13 \cdot 10^{-3}$ \\
128 -- $4\cdot 10^3$ -- $10^{-12}$ & $1.07 \cdot 10^{-6}$ & $1.69 \cdot 10^{-6}$ & $1.97 \cdot 10^{-5}$  & $7.60 \cdot 10^{-5}$ \\
128 -- $4\cdot 10^3$ -- $10^{-14}$ & $2.49 \cdot 10^{-7}$ & $1.34 \cdot 10^{-6}$ & $1.10 \cdot 10^{-5}$  & $6.30 \cdot 10^{-5}$ \\
128 -- $4\cdot 10^3$ -- $10^{-16}$ & $1.44 \cdot 10^{-8}$ & $1.33 \cdot 10^{-6}$ & $8.45 \cdot 10^{-6}$  & $6.26 \cdot 10^{-5}$ \\
\hline
128 -- $8\cdot 10^3$ -- $10^{-10}$ & $9.93 \cdot 10^{-7}$ & $1.45 \cdot 10^{-6}$ & $1.73 \cdot 10^{-4}$  & $2.87 \cdot 10^{-3}$ \\
128 -- $8\cdot 10^3$ -- $10^{-12}$ & $1.18 \cdot 10^{-6}$ & $1 .71 \cdot 10^{-6 }$ & $2.20 \cdot 10^{-5}$  & $6.92 \cdot 10^{-5}$ \\
128 -- $8\cdot 10^3$ -- $10^{-14}$ & $8.40 \cdot 10^{-7}$ & $1.38 \cdot 10^{-6}$ & $1.96 \cdot 10^{-5}$  & $6.48 \cdot 10^{-5}$ \\
128 -- $8\cdot 10^3$ -- $10^{-16}$ & $2.31 \cdot 10^{-8}$ & $1.34 \cdot 10^{-6}$ & $8.60 \cdot 10^{-6} $  & $6.07 \cdot 10^{-5}$ \\
\hline
128 -- $16\cdot 10^3$ -- $10^{-10}$ & $8.75 \cdot 10^{-7}$ & $1.71 \cdot 10^{-6}$ & $3.90 \cdot 10^{-4}$  & $3.29 \cdot 10^{-3}$ \\
128 -- $16\cdot 10^3$ -- $10^{-12}$ & $1.45 \cdot 10^{-6}$ & $1.67 \cdot 10^{-6}$ & $2.56 \cdot 10^{-5}$  & $6.79 \cdot 10^{-5}$ \\
128 -- $16\cdot 10^3$ -- $10^{-14}$ & $1.10 \cdot 10^{-6}$ & $1.69 \cdot 10^{-6}$ & $2.14 \cdot 10^{-5}$  & $6.79 \cdot 10^{-5}$ \\
128 -- $16\cdot 10^3$ -- $10^{-16}$ & $4.13 \cdot 10^{-8}$ & $1.35 \cdot 10^{-6}$ & $8.79 \cdot 10^{-6}$  & $5.98 \cdot 10^{-5}$ \\
\hline 
\hline
256 -- $4\cdot 10^3$ -- $10^{-10}$ & $ n.c. $ & $n.c. $ & $ n.c. $  & $n.c. $ \\
256 -- $4\cdot 10^3$ -- $10^{-12}$ & $1.03 \cdot 10^{-6}$ & $7.69 \cdot 10^{-7}$ & $1.99 \cdot 10^{-5}$  & $8.88 \cdot 10^{-5}$ \\
256 -- $4\cdot 10^3$ -- $10^{-14}$ & $2.54 \cdot 10^{-7}$ & $6.68 \cdot 10^{-7}$ & $1.13 \cdot 10^{-5}$  & $5.73 \cdot 10^{-5}$ \\
256 -- $4\cdot 10^3$ -- $10^{-16}$ & $2.54 \cdot 10^{-7}$ & $6.68 \cdot 10^{-7}$ & $1.12 \cdot 10^{-5}$  & $5.73 \cdot 10^{-5}$ \\
\hline
256 -- $8\cdot 10^3$ -- $10^{-10}$ & $9.28 \cdot 10^{-7}$ & $6.99 \cdot 10^{-7}$ & $1.61 \cdot 10^{-4}$  & $2.87 \cdot 10^{-3}$ \\
256 -- $8\cdot 10^3$ -- $10^{-12}$ & $9.77 \cdot 10^{-7}$ & $8.14 \cdot 10^{-7}$ & $1.93 \cdot 10^{-5}$  & $9.01 \cdot 10^{-5}$ \\
256 -- $8\cdot 10^3$ -- $10^{-14}$ & $8.35 \cdot 10^{-7}$ & $6.99 \cdot 10^{-7}$ & $1.93 \cdot 10^{-5}$  & $5.94 \cdot 10^{-5}$ \\
256 -- $8\cdot 10^3$ -- $10^{-16}$ & $1.55 \cdot 10^{-8}$ & $6.72 \cdot 10^{-7}$ & $8.52 \cdot 10^{-6}$  & $5.49 \cdot 10^{-5}$ \\
\hline
256 -- $16\cdot 10^3$ -- $10^{-10}$ & $5.02 \cdot 10^{-7}$ & $7.87 \cdot 10^{-7}$ & $6.46 \cdot 10^{-4}$  & $3.32 \cdot 10^{-3}$ \\
256 -- $16\cdot 10^3$ -- $10^{-12}$ & $1.64 \cdot 10^{-6}$ & $8.14 \cdot 10^{-7}$ & $2.73 \cdot 10^{-5}$  & $6.83 \cdot 10^{-5}$ \\
256 -- $16\cdot 10^3$ -- $10^{-14}$ & $1.37 \cdot 10^{-6}$ & $7.68 \cdot 10^{-7}$ & $2.35 \cdot 10^{-5}$  & $6.24 \cdot 10^{-5}$ \\
256 -- $16\cdot 10^3$ -- $10^{-16}$ & $3.04 \cdot 10^{-8}$ & $6.75 \cdot 10^{-7}$ & $8.69 \cdot 10^{-6}$  & $5.41 \cdot 10^{-5}$ \\
\end{tabular}
\label{tab:LandauDampingErrors_2}
\end{table}
 
\begin{figure}
\begin{center}
 \includegraphics[scale=0.5=\textwidth]{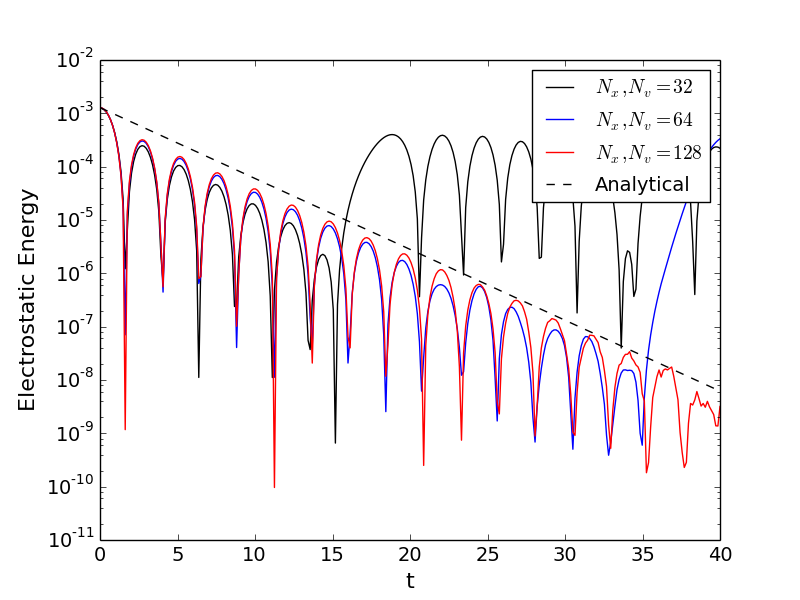}
 \caption{Linear Landau damping testcase (see section \ref{subsec:LandauDamping}). Electrostatic energy as function of time for different resolutions in the phase space. Dash line is the analytical expected decay for the electrostatic energy.}
 \label{fig:LandauDamping}
 \end{center}
\end{figure}

The results of the numerical testcases are reported in Table~\ref{tab:LandauDampingErrors_1}. The numerical experiments show that the conservation of mass,
momentum and hamiltonian are well respected for all the discretizations adopted. Concerning the error with respect to the reference simulation, it has been observed that the error is
dominated by the space-time discretization. In particular, using a residual tolerance ($\varepsilon$) too low with a given discretization does not allow to improve the results.
On the other hand, when refining the mesh or when using a small $\Delta t$, a high tolerance may result in a non-convergence of the solution. In Figure~\ref{fig:LandauDamping} the decay in
electrostatic energy is shown as a function of time for $N_x=N_v=(32,64,128)$, compared to the theoretical decay. The behavior in terms of decay and of Langmuir frequency is in agreement with the results presented in the litterature.

\medskip

Let us mention here that the memory needed to store a rank-$n$ function is $n(N_x+ N_v)$, which has to be compared with $N_xN_v$, the total number of degrees of freedom in the system. 
The evolution in time of the ranks of the approximation of $f$ computed by the approach is plotted in Figure~\ref{fig:ranksLandau} for the following discretization parameters: $N_x = N_v = 512$, 
$\epsilon = 10^{-16}$, $N_t = 32 000$, and $T = 10$. We observe that the maximal rank of the approximation is obtained at the final time of the simulation and is approximately equal to $n =50$. 
The worst compression factor $\frac{N_x N_v}{n(N_x + N_v)} \approx 5$ remains reasonable in this 1d case.
We observe numerically an interesting trend: the rank seems to increase linearly with time and are independent of $N_x$ and $N_v$.

\begin{figure}
\begin{center}
\includegraphics[scale=0.5=\textwidth]{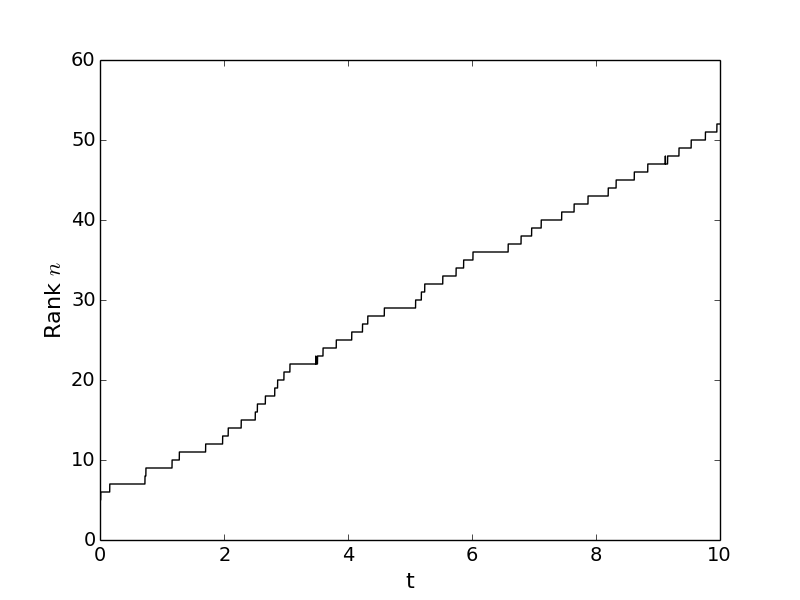}
\caption{Evolution in time of the rank of the approximation of $f$.}
\label{fig:ranksLandau}
\end{center}
\end{figure}

\subsection{Two stream instability.}
\label{subsec:stream1D}
We present the classical 1D-1D two stream-instability testcase. The domain is $\Omega = \Omega_x \times \Omega_v = [0, 10\pi /\omega] \times [-10,10]$. 
\begin{figure}[!tbp]
  \centering
  \begin{minipage}[b]{0.8\textwidth}
    \includegraphics[width=\textwidth]{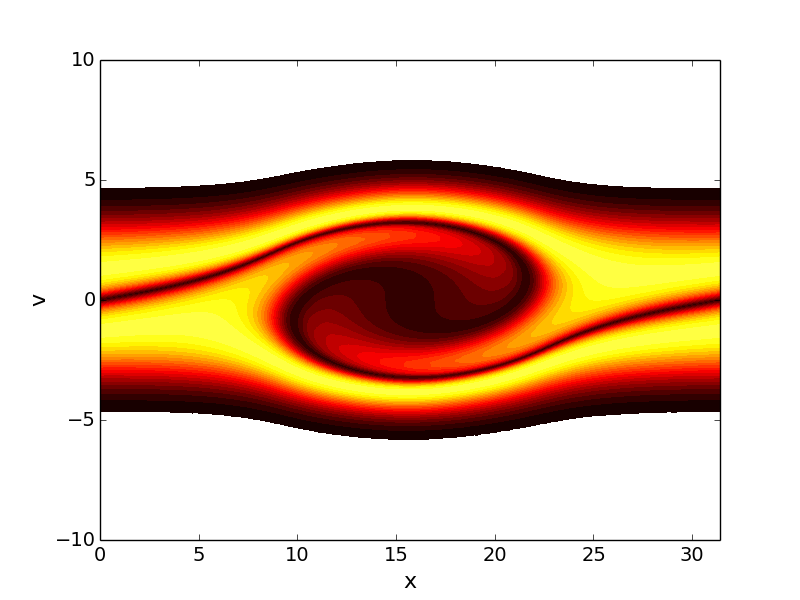}
  \end{minipage}
  \hfill
  \begin{minipage}[b]{0.8\textwidth}
    \includegraphics[width=\textwidth]{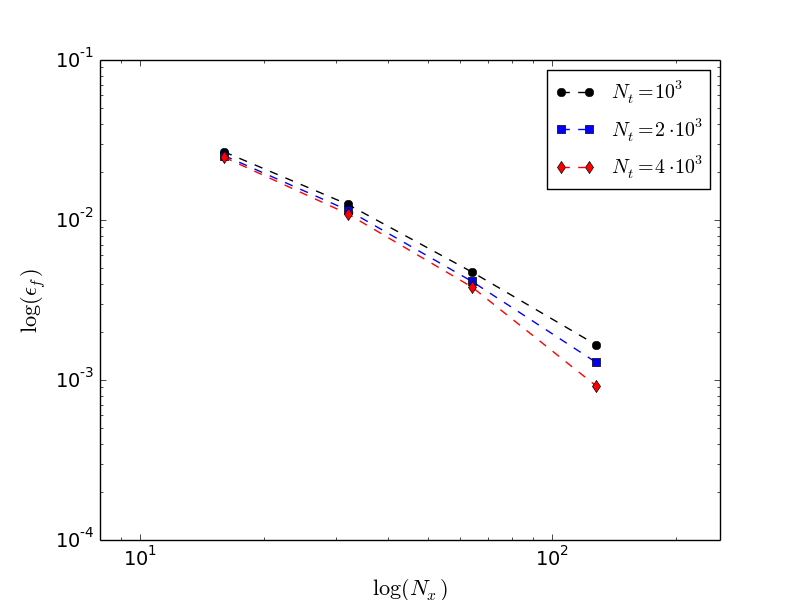}
  \end{minipage}
  \caption{Two stream instability testcase (section \ref{subsec:stream1D}): a) Contours of the reference solution (black for the lowest value) at final time and b) Errors with respect to a reference simulation as function of the phase space discretization, for different time steps.}
\label{fig:stream1D}
\end{figure}

The final time of the evolution is $T = 36.0$. The initial condition has the following form:
\begin{eqnarray}
f(x,v; t=0) = F(x)G(v), \\
F(x) = 1 + \beta \cos(k x), \\
G(v) = \frac{1}{\sqrt{4\pi}} \exp \left( - \frac{(v-v_0)^2}{2} \right) + \frac{1}{\sqrt{4\pi}} \exp \left( - \frac{(v+v_0)^2}{2} \right),
\end{eqnarray}
where $v_0 = 2.4$ and $\beta = 10^{-3}$. A mixed discretization is considered, namely a spectral collocation method for the space and standard centered finite differences in velocity. The contour 
plot of the reference solution at final time is shown in Figure~\ref{fig:stream1D}. The conservation properties and the errors with respect to a reference simulation 
are investigated by varying the phase space discretization as well as time step and the residual tolerance. The results are very similar to the ones obtained for the linear Landau damping 
testcase. For the sake of brevity, the conservation error properties are not reported. The errors with respect to a reference simulation 
($N_x = 256, \ N_v = 512, \ N_t = 8\cdot 10^3, \varepsilon = 10^{-14}$) are computed by varying the discretization of the phase space and the time step. 
In particular, $N_x$ ranges in $[16,32,64,128]$, $N_v = 2 N_x$ and $N_t = [10^3, 2\cdot 10^3, 4\cdot 10^3]$. The tolerance on the residual is varied and the errors when considering $\
\varepsilon=10^{-12}$ are shown in Figure~\ref{fig:stream1D}. A second order convergence rate is retrieved for the space discretization, at fixed 
time step. Whereas the error is relatively insensitive to the time step when a coarse discretization is considered, a definite dependence is seen for the finest grid resolution. 
This is due to the fact that, on the coarse grids, the discretization error is dominated by the space discretization error.

\medskip
%
%
%

\subsection{2D-2D simulations}

In this section, we present a 2D-2D Landau damping test case. The simulation domains are $\Omega_x = (0,4\pi)^2$ and $\Omega_v = (-10,10)^2$. We impose as before periodic boundary conditions on $\Omega_x$ 
and homogeneous Dirichlet boundary conditions on $\Omega_v$. Uniform tensor discretizations are used for $\Omega_x$ and $\Omega_v$, and two different simulations are obtained for the following numbers of degrees 
of freedom: $(N_x, N_v) = [ (16^2, 32^2), (32^2, 64^2)]$. The error tolerance criterion of the algorithm is set to be $\epsilon = 10^{-15}$. Time step is equal to $\Delta t  = 2.10^{-4}$. 

The initial condition is defined as
$$
f_0(x,v) = \frac{1}{\sqrt{2\pi}^3}\left[1 - \beta \sin(\omega x_1) - \beta\sin(\omega x_2)\right] \exp( -\frac 1 2 (v_1^2 + v_2^2)),
$$
where $\beta = 0.01$ and $\omega = 0.5$.

\medskip

The evolution of the electric energy as a function of time is shown in Figure~\ref{fig:Landau2D2D} for the two different discretizations mentioned above. It can be seen 
that these are in agreement with the predicted analytical decay. Conservation properties of mass, momentum and total energy behave similarly to 1D-1D cases. Ranks of the approximated solution obtained by the 
algorithm also seem to increase linearly with time.

\begin{figure}
\centering
\includegraphics[scale=0.5=\textwidth]{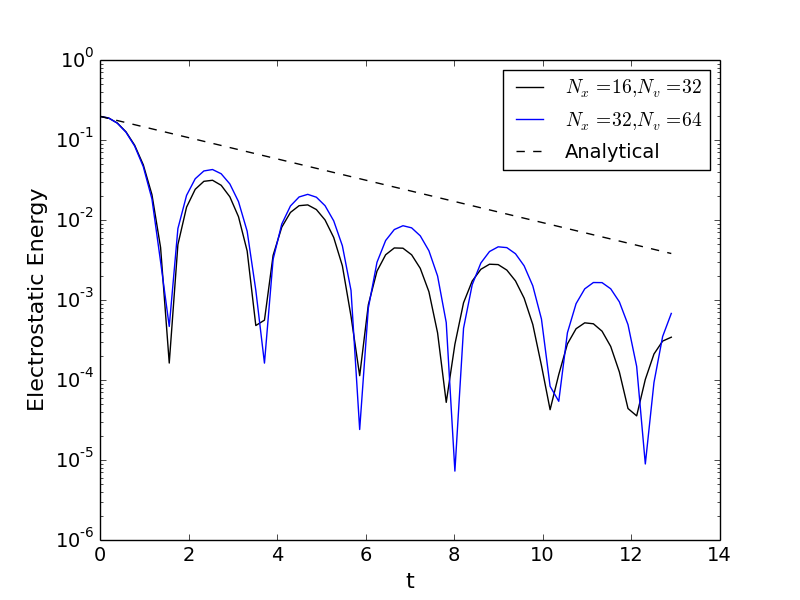} 
\caption{Evolution of the electric energy as a function of time in the 2D Landau damping test case.}
\label{fig:Landau2D2D}
\end{figure}

\medskip

We mention here that encouraging preliminary results have been obtained on 3D-3D test cases. Parallelisation of the method, which is needed to reduce the computational cost, 
is work in progress, and should enable to obtain results in more realistic settings.



\section{Conclusion}
In this work a dynamical adaptive tensor method has been proposed to build parsimonious discretizations for the Vlasov-Poisson system. 
It allows to treat generic geometries and can be applied to generic heterogeneous discretizations in space and velocity, making it 
a flexible tool for the simulation of kinetic equation within an Eulerian framework.
The method is dynamical in time and the time advancing is design to preserve the Hamiltonian character of the system, 
with a second order accuracy. Several testcases were proposed to validate the method and assess its properties.

Several perspectives arise, concerning the parallelisation of the method (which is mandatory to deal with more realistic 3D-3D settings) 
and its extension to other kinetic equations, involving collision operators. These will be the object of a further investigation.

\section*{Appendix A: Proof of Proposition~\ref{prop:FPPGD}}\label{sec:app2}

\begin{proof}[Proof of Proposition~\ref{prop:FPPGD}]
Let us denote by $b\in H$ the Riesz representative of $\bfb$ in $H$. The element $f$ solution of (\ref{eq:nonsym}) is then the unique solution to 
$$
(I + \widetilde{A})f  =b,
$$
where $I$ denotes the identity operator on $H$. By assumption, $\|\widetilde{A}\|_{\cL(H;H)} \leq M \kappa$, thus both (A1) or (A2) imply that $\|\widetilde{\bfa}\|_{\cL(H\times H ;\R)} < 1$. 
For all $n\in\N$, let us denote by $R_n: = b - (I + \widetilde{A})f_n$ the residual of the equation in $H$ after $n$ iterations of the Fixed-point PGD algorithm. Since 
$(r_{n+1}, s_{n+1}) \in H_x \times H_v$ is solution to the minimization problem (\ref{eq:PGDit3}), it satisfies
\begin{equation}\label{eq:newpb}
(r_{n+1}, s_{n+1}) \in \mathop{\mbox{\rm argmin}}_{(r,s)\in H_x \times H_v} \|R_n - r\otimes s\|_H^2. 
\end{equation}
Thus, we have the following properties on the tensor product function $r_{n+1}\otimes s_{n+1}$
\begin{eqnarray}
& (r_{n+1}, s_{n+1}) \in \mathop{\mbox{argmax}}_{(r,s) \in H_x \times H_v} \frac{\langle R_n , r\otimes s \rangle_H}{\|r\otimes s\|_H},   \label{eq:prop1}\\
& \|r_{n+1} \otimes s_{n+1}\|_H  = \mathop{\max}_{(r,s) \in H_x \times H_v} \frac{\langle R_n , r\otimes s \rangle_H}{\|r\otimes s\|_H}, \label{eq:prop2} \\
& \langle R_n - r_{n+1}\otimes s_{n+1}, r_{n+1}\otimes s_{n+1} \rangle_H = 0. \label{eq:prop3} \nonumber
\end{eqnarray}

We refer the reader to~\cite{le2009results,cances2011,falco2011, falco2012, figueroa2012} for a proof of the properties (\ref{eq:prop1}), (\ref{eq:prop2}) and (\ref{eq:prop3}), which are consequences of (\ref{eq:newpb}). Let us already point out here that 
if $H= H_x\otimes H_v$ (which is the case when assumption (A1) is satisfied), we have in addition 
\begin{equation}\label{eq:prop4}
 \mathop{\max}_{(r,s) \in H_x \times H_v} \frac{\langle R_n - r_{n+1}\otimes s_{n+1}, r\otimes s \rangle_H}{\|r\otimes s\|_H}\leq  \|r_{n+1} \otimes s_{n+1}\|_H. 
\end{equation}

\medskip

Thus, since $f_{n+1} = f_n + r_{n+1} \otimes s_{n+1}$,  
\begin{eqnarray*}
\|R_n\|^2_H - \|R_{n+1}\|_H^2 & =  &\|R_n\|_H^2 - \|R_n - r_{n+1}\otimes s_{n+1} - \widetilde{A} r_{n+1}\otimes s_{n+1} \|_H^2, \\
& = & \|R_n\|_H^2 - \|R_n - r_{n+1}\otimes s_{n+1}\|^2 - \|\widetilde{A} r_{n+1}\otimes s_{n+1} \|_H^2 \\
&& + 2 \langle R_n - r_{n+1}\otimes s_{n+1}, \widetilde{A} r_{n+1}\otimes s_{n+1}\rangle_H,\\
& = & \|R_n\|_H^2 - \|R_n\|_H^2  +  \|r_{n+1}\otimes s_{n+1}\|_H^2 - \|\widetilde{A} r_{n+1}\otimes s_{n+1} \|_H^2 \\
&& + 2 \langle R_n 
 - r_{n+1}\otimes s_{n+1}, \widetilde{A} r_{n+1}\otimes s_{n+1}\rangle_H, \quad \mbox{ (using (\ref{eq:prop3}))}\\
& = & + \|r_{n+1}\otimes s_{n+1}\|_H^2 - \|\widetilde{A} r_{n+1}\otimes s_{n+1} \|_H^2 \\
&& + 2 \sum_{\mu=1}^M \langle R_n - r_{n+1}\otimes s_{n+1}, (A_x^\mu r_{n+1}) \otimes (A_v^\mu s_{n+1})\rangle_H,\\
& \geq & (1 - \kappa M) \|r_{n+1}\otimes s_{n+1}\|_H^2  + 2 \sum_{\mu=1}^M \langle R_n - r_{n+1}\otimes s_{n+1}, (A_x^\mu r_{n+1}) \otimes (A_v^\mu s_{n+1})\rangle_H.\\
\end{eqnarray*}
At this point, we treat the two cases separately. Let us first assume that (A1) holds. Then, 
\begin{eqnarray*}
\|R_n\|^2_H - \|R_{n+1}\|_H^2 & \geq & (1 - \kappa M) \|r_{n+1}\otimes s_{n+1}\|_H^2 \\
&& - 2 \sum_{\mu=1}^M \|r_{n+1}\otimes s_{n+1}\|_H \| (A_x^\mu r_{n+1}) \otimes (A_v^\mu s_{n+1})\|_H, \quad \mbox{(using (\ref{eq:prop4}))}\\
 & = & (1 - \kappa M) \|r_{n+1}\otimes s_{n+1}\|_H^2 \\
 && - 2 \sum_{\mu=1}^M \|r_{n+1}\otimes s_{n+1}\|_H \| (A_x^\mu \otimes A_v^\mu)( r_{n+1} \otimes s_{n+1} )\|_H, \\
 & \geq & (1 - \kappa M) \|r_{n+1}\otimes s_{n+1}\|_H^2  - 2 \sum_{\mu=1}^M \kappa \|r_{n+1}\otimes s_{n+1}\|_H^2 , \\
  & \geq & (1 - 3\kappa M) \|r_{n+1}\otimes s_{n+1}\|_H^2.\\
\end{eqnarray*}
Assume now that (A2) holds. Then, 
\begin{eqnarray*}
\|R_n\|^2_H - \|R_{n+1}\|_H^2 & \geq & (1 - \kappa M) \|r_{n+1}\otimes s_{n+1}\|_H^2 \\
&& - 4 \sum_{\mu=1}^M \|r_{n+1}\otimes s_{n+1}\|_H \| (A_x^\mu r_{n+1}) \otimes (A_v^\mu s_{n+1})\|_H, \quad \mbox{(using (\ref{eq:prop2}))}\\
 & = & (1 - \kappa M) \|r_{n+1}\otimes s_{n+1}\|_H^2  - 4 \sum_{\mu=1}^M \|r_{n+1}\otimes s_{n+1}\|_H \| (A_x^\mu \otimes A_v^\mu)( r_{n+1} \otimes s_{n+1} )\|_H, \\
 & \geq & (1 - \kappa M) \|r_{n+1}\otimes s_{n+1}\|_H^2  - 4 \sum_{\mu=1}^M \kappa \|r_{n+1}\otimes s_{n+1}\|_H^2 , \\
  & \geq & (1 - 5\kappa M) \|r_{n+1}\otimes s_{n+1}\|_H^2.\\
\end{eqnarray*}
In both cases, there exists a constant $\eta>0$ such that 
$$
\|R_n\|^2_H - \|R_{n+1}\|_H^2 \geq \eta \|r_{n+1}\otimes s_{n+1}\|_H^2.
$$
Thus, the sequence $(\|R_n\|_H^2)_{n\in \N}$ is non-increasing and converges. Since $\|\widetilde{A}\|_{\cL(H;H)} < 1$, and $R_n = b - (\widetilde{A} + I)f_n$, this implies that $(f_n)_{n\in\N}$ is a bounded sequence in $H$. 
Besides, the series $\sum_{n\in\N^*} \|r_n\otimes s_n\|_H^2$ is convergent, and $\dps \|r_n\otimes s_n\|_H \mathop{\longrightarrow}_{n\to +\infty} 0$. Up to the extraction of a subsequence 
(still denoted $(f_n)_{n\in\N}$ for the sake of simplicity), 
$(f_n)_{n\in \N}$ weakly converges in $H$ to some $g \in H$. Property (\ref{eq:prop2}) implies that 
$$
\forall (r,s)\in H_x \times H_v, \quad \mid \langle R_n, r\otimes s \rangle_H \mid  = \mid \langle b - (I + \widetilde{A})f_n , r\otimes s \rangle_H \mid  \leq \|r_{n+1}\otimes s_{n+1}\|_H \|r\otimes s\|_H. 
$$
Since $\dps b - (I + \widetilde{A})f_n \mathop{\rightharpoonup}_{n\to +\infty} b - (I + \widetilde{A})g$, we obtain that $g$ is necessarily equal to $f$, the unique solution of (\ref{eq:nonsym}). 
The sequence $(f_n)_{n\in\N}$ thus entirely converges (weakly) to $f$ in $H$. The strong convergence can be obtained using the same arguments as in~\cite{cances2011}, which yields the desired result.
\end{proof}

\bibliographystyle{plain}
\bibliography{biblioVlasov.bib}

\end{document}